\documentclass[10pt]{amsart}

\usepackage{amsmath,amsthm,amssymb,amscd,caption,enumerate,float,mathdots,enumerate,shuffle,hyperref,mathtools}
\usepackage{setspace,mathrsfs}
\usepackage{tikz}
\usetikzlibrary{decorations.pathreplacing,angles,quotes}
\usetikzlibrary{calc}
\usetikzlibrary{arrows}
\usepackage{ytableau}

\allowdisplaybreaks

\newtheorem{theorem}{Theorem}[section]
\newtheorem{proposition}[theorem]{Proposition}

\newtheorem{lemma}[theorem]{Lemma}
\newtheorem{corollary}[theorem]{Corollary}

\numberwithin{equation}{section}

\theoremstyle{definition}
\newtheorem{definition}[theorem]{Definition}
\newtheorem{example}[theorem]{Example}
\newtheorem{remark}[theorem]{Remark}
\newcommand{\Q}{\mathbb{Q}}
\newcommand{\C}{\mathbb{C}}
\newcommand{\Z}{\mathbb{Z}}
\newcommand{\R}{\mathbb{R}}

\newcommand{\N}{\Z_{\geq 1}}

\newcommand{\kk}{{\bf k}}

\newcommand{\lsm}{\lambda \slash \mu}

\newcommand{\sfm}{S_M^{f}}

\newcommand{\zr}{\zeta^{\rm{reg}}}

\DeclareMathOperator{\SSYT}{\operatorname{SSYT}}

\newcommand{\adots}{\raisebox{-2pt}{$\iddots$}} 
\newcommand{\svdots}{\raisebox{-2pt}{$\vdots$}} 
\newcommand{\vsmall}{\rotatebox[origin=c]{-90}{$<$}}

\definecolor{gray}{RGB}{230,230,230}  % brighter gray
\definecolor{darkgray}{RGB}{130,130,130}  % brighter darkgray

\definecolor{gray1}{RGB}{138, 237, 226}
\definecolor{gray3}{RGB}{240, 162, 132}
\definecolor{gray2}{RGB}{132, 240, 145}
\definecolor{gray4}{RGB}{240, 236, 132}

\title[Generalized Jacobi-Trudi determinants and evaluations of Schur MZVs]{Generalized Jacobi-Trudi determinants \\and evaluations of Schur multiple zeta values}

\author{Henrik Bachmann}

\address{Graduate School of Mathematics,  Nagoya University, Nagoya, Japan.}
\email{henrik.bachmann@math.nagoya-u.ac.jp}

\author{Steven Charlton}

\address{Max Planck Institute for Mathematics, Vivatsgasse 7,
	Bonn 53111, Germany.}
\email{stven.charlton@gmail.com}


\ytableausetup{centertableaux, boxsize=1.2em}

\begin{document}
	
	\begin{abstract}
		We present new determinant expressions for regularized Schur multiple zeta values. These generalize the known Jacobi-Trudi formulae and can be used to quickly evaluate certain types of Schur multiple zeta values. Using these formulae we prove that every Schur multiple zeta value with alternating entries in $1$ and $3$ can be written as a polynomial in Riemann zeta values. Furthermore, we give conditions on the shape, which determine when such Schur multiple zetas are polynomials purely in odd or in even Riemann zeta values.
	\end{abstract}

	\date{14 August 2019}
	\maketitle
	
	\section{Introduction}
	The purpose of this paper is to give a generalized Jacobi-Trudi determinant expression for regularized Schur multiple zeta values. Using this determinant expression we obtain new explicit evaluations of certain Schur multiple zeta values. This generalizes results given in \cite{NPY} and answers questions posed in \cite{BY}.
	
	The literature on symmetric functions contains many determinantal results for (skew) Schur functions: two types of Jacobi-Trudi determinants \cite{M}, the Giambelli determinant \cite{G}, Lascoux and Pragacz's rim ribbon determinant \cite{LP} and the determinant expression of Hamel and Goulden \cite{HG} based on outside decompositions, which generalizes all the aforementioned ones. 
	
	The proofs of all of these determinant expressions are based on the Gessel-Viennot methods \cite{GV}, in which one constructs lattice graphs, assigns weights to their edges and then writes the weights of disjoint path systems as determinants of matrices whose entries are given by the weights of individual paths. The weights of path systems then corresponds to  Schur functions $s_{\lsm}$, where $\lsm$ is an arbitrary skew Young diagram,  and the entries in the matrix correspond to Schur functions of so-called ribbons. In the case of Jacobi-Trudi determinants, these ribbons are given by columns or by rows, which then give the complete homogeneous or elementary symmetric functions respectively. This, in particular, gives the well-known result that every Schur function can be written as a polynomial expression in complete homogeneous functions or elementary symmetric functions. 
	
	Usually, the weights of edges in graphs appearing in these proofs depend only on the horizontal position in the graph. In \cite{NPY} it was observed, that these proofs can also be adapted to the case of Schur multiple zeta values by, roughly speaking, making the weights of edges also depend on their horizontal position. Schur multiple zeta values can be viewed as a variation of Schur functions, which assign to a Young Tableau $\kk$ (that is, a Young diagram filled with integers in each box) the real number $\zeta(\kk)$. For integers $a,c,d,e\geq 1,\,b,f \geq 2$ the following sum is an example of a Schur multiple zeta value    
	\begin{equation} \label{eq:3stair}
	\zeta\left(\ {\footnotesize \ytableausetup{centertableaux, boxsize=1.2em}
		\begin{ytableau}
		\none & a& b \\
		c& d\\
		e & f
		\end{ytableau}}\ \right)
	= \sum_{{\scriptsize
			\arraycolsep=1.4pt\def\arraystretch{0.8}
			\begin{array}{ccccc}
			&&m_a&\leq&m_b \\
			&&\vsmall&& \\
			m_c&\leq&m_d&& \\
			\vsmall&&\vsmall&& \\
			m_e&\leq&m_f&&
			\end{array} }} \frac{1}{m_a^{\,a} \,\, m_b^b \,\,  m_c^c \,\,  m_d^d \,\,  m_e^e \,\,  m_f^f}  \,,
	\end{equation}
	where we assume the summation indices $m_\ast$ range over positive integers. Schur multiple zeta values, which were first introduced by Y. Yamasaki, also generalize the classical multiple zeta(-star) values, which for $k_1,\dots,k_{r-1} \geq 1, k_r \geq 2$ are defined by
	\begin{equation} \label{eq:mzv}
	\begin{split}
	\zeta(k_1,\ldots,k_r)=\sum_{0<m_1<\cdots<m_r} \frac{1}{m_1^{k_1}\cdots m_r^{k_r}} \,,\quad
	\zeta^\star(k_1,\ldots,k_r)=\sum_{0<m_1\leq\cdots \leq m_r} \frac{1}{m_1^{k_1}\cdots m_r^{k_r}} \,.
	\end{split}
	\end{equation}
	Schur multiple zeta values specialize to these numbers by choosing columns or choosing rows as the Young tableaux, respectively.
	
	In \cite{NPY} the authors give Jacobi-Trudi formulae for Schur multiple zeta values of Young tableaux with constant diagonal entries and show that these can be written as determinants whose entries are given by the multiple zeta(-star) values \eqref{eq:mzv}.
	We will generalize these formulae and show (Theorem \ref{thm:jacobiformzv}) that there exist explicit determinant expression for these Schur multiple zeta values as determinants in Schur multiple zeta values of subribbons of a fixed ribbon $R$. 
	Choosing the ribbon $R$ to a row or a column then specializes to the Jacobi-Trudi formula proven in \cite{NPY}. For example, if we choose the ribbon $\ytableausetup{centertableaux, boxsize=0.5em} R = {\footnotesize 
		\begin{ytableau}
		\none & \,& \,\\
		\none & \,\\
		\,& \,
		\end{ytableau}}$ then the  Schur multiple zeta value \eqref{eq:3stair} (where we  set $f=c$ to make the entries constant on the diagonals) can be written as
	\ytableausetup{centertableaux, boxsize=1.2em}
	\begin{equation*} 
	\zeta\left(\ {\footnotesize 
		\begin{ytableau}
		\none & a& b \\
		c& d\\
		e & c
		\end{ytableau}}\ \right) = 
	\det
	\begin{pmatrix}
	\zeta\left(\ {\footnotesize 
		\begin{ytableau}
		\none & a& b \\
		\none & d\\
		e & c
		\end{ytableau}}\ \right)
	& 
	\zeta\left(\ {\footnotesize 
		\begin{ytableau}
		a& b \\
		d\\
		c
		\end{ytableau}}\ \right)
	\\[20pt]
	\zeta\left(\ {\footnotesize 
		\begin{ytableau}
		e & c
		\end{ytableau}}\ \right)
	&
	\zeta\left(\ {\footnotesize 
		\begin{ytableau}
		c
		\end{ytableau}}\ \right)
	\end{pmatrix}\,.
	\end{equation*}
	We will prove these determinant expressions in Theorem \ref{thm:jacobitrudi} for a more general object $\sfm$ (Definition \ref{def:genschurpol}), which generalizes (truncated) Schur multiple zeta values and Schur polynomials simultaneously. By extending the notion of stuffle regularized multiple zeta values to the Schur case, we will obtain a generalized Jacobi-Trudi determinant expression for regularized Schur multiple zeta values in Theorem \ref{thm:jacobiformzv}.
	
	One nice application of this generalized determinant expression, which was also the original motivation for this project, is the evaluation of Schur multiple zeta values of Young tableaux with alternating entries in $1$ and $3$. These were studied in detail in the work \cite{BY}, where the authors consider in general so-called ``Checkerboard style'' Schur multiple zeta values. There it was shown that the stairs with alternating entries of $1$ and $3$ evaluate in an especially nice way, to multiples of odd single zeta values, e.g.
	\begin{equation}\label{eq:13stariexample}
	\zeta\left(\ {\footnotesize \begin{ytableau}
		1 & *(gray)3  \\
		*(gray)3
		\end{ytableau}}\ \right) 
	= \frac{1}{4} 
	\zeta(7) \,,\qquad
	\zeta\left(\ {\footnotesize \begin{ytableau}
		\none & \none & 1 \\
		\none & 1 & *(gray)3\\
		1 & *(gray)3
		\end{ytableau}}\ \right)  = \frac{1}{8} \zeta(9) \,. 
	\end{equation}
	We will show that any Schur multiple zeta values with alternating entries in $1$ and $3$ can be written as a polynomial in single zeta values (Theorem \ref{thm:13zeta}). For example, we obtain for the $3\times3$ square the evaluation
	\begin{align}\label{eq:3x3square13}
	\zeta\left(\ {\footnotesize
		\begin{ytableau}
		*(gray)3 & *(white)1 & *(gray)3\\
		*(white)1 & *(gray)3& *(white)1\\
		*(gray)3 & *(white)1 & *(gray)3\\
		\end{ytableau}}\ \right) &=\frac{1}{32} \det
	\begin{pmatrix}
	\zeta(3) & \frac{\pi^4}{180} & \zeta(7)\\
	\frac{\pi^4}{72} & \zeta(5) & \frac{17\pi^8}{90720} \\
	\zeta(7) &  \frac{13\pi^8}{226800} & \zeta(11)
	\end{pmatrix}\,.
	\end{align}
	
	Choosing the ribbon $R$ to be a stair like in \eqref{eq:13stariexample} and using our generalized Jacobi-Trudi determinant, we will further give explicit conditions for in which cases such Schur multiple zetas are polynomials in  $\pi^4$ or polynomials in only odd single zeta values (Corollary \ref{cor:13tessel}). Similar statements are also be obtained for the case of Schur multiple zeta values with alternating entries in $1$ and $2$ (Corollary \ref{cor:12tessel}).
	
	\subsection*{Acknowledgement} This work started during the Trimester Program `Periods in Number Theory, Algebraic Geometry, and Physics' at the Hausdorff Research Institute for Mathematics in Bonn. The authors would like to thank the organizers of this program and the Max-Planck-Institut f\"ur Mathematik in Bonn for hospitality and support. This project was partially supported by JSPS KAKENHI Grant 19K14499.
	
	\section{Jacobi-Trudi determinants}
	
	The goal of this section is to prove a generalized Jacobi-Trudi type determinant formula for a generalization of Schur polynomials (Definition \ref{def:genschurpol}).
	A partition is a tuple $\lambda = (\lambda_1,\dots,\lambda_h)$ of positive integers
	$\lambda_1 \geq \dots \geq \lambda_h \geq 1$.
	For another partition $\mu=(\mu_1,\dots,\mu_r)$ we write $\mu \subset \lambda$ if $r\leq h$ and $\mu_j \leq \lambda_j$ for $j=1,\dots,r$.  
	For partitions $\lambda,\mu$ with $\mu \subset \lambda$ we identify the pair $\lambda\slash\mu=(\lambda,\mu)$ with its (skew) Young diagram
	\[D(\lambda \slash \mu) = \left\{(i,j) \in \Z^2 \mid 1 \leq i \leq h\,, \mu_i < j \leq \lambda_i \right\},\]
	where we set $\mu_j = 0$ for $j>r$.

	A Young tableau $\kk =  (k_{i,j})_{(i,j) \in D(\lambda \slash \mu)}$ of shape $\lambda \slash \mu$ is a filling of $D(\lambda\slash\mu)$ obtained by putting $k_{i,j}\in\N$ into the $(i,j)$-entry of $D(\lambda\slash\mu)$. For shorter notation we will also just write $(k_{i,j})$ in the following if the shape $\lambda \slash \mu$ is clear from the context. 
	For example when $\lambda\slash\mu=(5,4,3)\slash (3,1)$ we visualize this Young tableau as
	\[
	\kk \,\, = \,\, (k_{i,j}) \,\, = \,\, {\footnotesize \ytableausetup{centertableaux, boxsize=2em}
		\begin{ytableau}
		\none & \none & \none & k_{1,4} & k_{1,5} \\
		\none & k_{2,2} & k_{2,3} & k_{2,4} \\
		k_{3,1} & k_{3,2} & k_{3,3} 
		\end{ytableau}}\,. 
	\]
	A Young tableau $(m_{i,j})$ is called semi-standard if $m_{i,j}<m_{i+1,j}$ and $m_{i,j}\leq m_{i,j+1}$ for all $i$ and $j$.
	The set of all Young tableaux and all semi-standard Young tableaux of shape $\lambda \slash \mu$ are denoted by $T(\lambda \slash \mu)$ and 
	$\SSYT(\lambda \slash \mu)$, respectively. For $M \in \Z_{>0}$ we further define the set of restricted semi-standard Young tableaux by
	\[\SSYT_M(\lambda \slash \mu) = \big\{ (m_{i,j}) \in  \SSYT(\lambda \slash \mu) \mid m_{i,j} < M \big\}\,. \]
	
	\begin{definition} \label{def:genschurpol}Let $A$ be a ring. For a map 
		\[ f: \Z_{>0} \times  \Z_{>0} \rightarrow A\,, \]
		a skew diagram $\lsm$ , a Young tableaux $\kk =(k_{i,j})\in T(\lsm)$ and an integer $M\in \Z_{>0}$ we define
		\begin{equation}\label{eq:defs}
		\sfm(\kk) = \sum_{(m_{i,j}) \in \SSYT_M(\lambda \slash \mu)} \,\, \prod_{(i,j) \in D(\lambda \slash \mu)} f(m_{i,j}, k_{i,j}) 
		\end{equation}
		and set $\sfm(\emptyset) = 1$.
	\end{definition}
	
	This object generalizes the classical skew Schur polynomials\footnote{Which give the Schur functions in the case $M \rightarrow \infty$.} $s_{\lsm}(x_1,\dots,x_M)$ in the case $A=\Q[x_1,\dots,x_M]$, $f(m,d)=x_m$ and it gives the truncated Schur multiple zeta values $\zeta_M(\kk)$ in the case $A=\R$, $f(m,d) = m^{-d}$ (see \eqref{eq:defschurmzv}). The latter will be discussed in detail in Section \ref{sec:smzv}. In \cite{HG} the authors proved a generalized Jacobi-Trudi formula for Schur functions, which we will extend in this section to the above functions $\sfm$. For this, we will need to introduce some notation from \cite{HG}. A \emph{ribbon}, also called a strip in \cite{HG}, is a skew diagram which does not contain a $2\times 2$ square. Clearly, every skew diagram can be decomposed, in various ways, into a disjoint union of ribbons. We will be interested in a specific decomposition in which all the starting and ending boxes of the ribbons are attached to the outside of the skew diagram.
	
	\begin{definition}
		Suppose that $\theta_1,\dots,\theta_n$ are ribbons in a Young diagram $\lsm$ and each ribbon has a starting box on the left or bottom perimeter of the diagram and an ending box on the right or top perimeter of the diagram. Then if the disjoint union of these ribbons equals  $\lsm$ we say that $\Theta=(\theta_1,\dots,\theta_n)$ is a \emph{outside decomposition} of $\lsm$.
	\end{definition}

	\begin{example} \label{ex:outdec} For the case $\lsm = (4,3,3,2,1)\slash (1)$ a possible outside decomposition of $\lsm$ is given by $\Theta= (\theta_1,\theta_2,\theta_3,\theta_4)$ as follows. 
		
		\begin{center}
			\ytableausetup{centertableaux, boxsize=1.2em}
			\begin{minipage}[c]{0.5\linewidth}
				\begin{align*}
				\lsm = \begin{ytableau}
				\none & *(gray4)\,  & *(gray3)\,  & *(gray3)\, \\
				*(gray3)\,  & *(gray3)\, & *(gray3)\,  \\
				*(gray3) \,  & *(gray2) \,  & *(gray2) \,  \\
				*(gray2) \,  & *(gray2) \,   \\
				*(gray1) \,   \\
				\end{ytableau}
				\end{align*}    
			\end{minipage}
			\hspace{-3em}
			\ytableausetup{centertableaux, boxsize=0.8em}
			\begin{minipage}[c]{0.5\linewidth}
				\begin{align*}
				\theta_4 &{}= (2)\slash(1) = \begin{ytableau} *(gray4) \, \end{ytableau}\\     \theta_3 &{}= (4,3,1)\slash (2) =  \begin{ytableau} \none &  \none& *(gray3) \, & *(gray3) \,  \\ *(gray3) \, & *(gray3) \,  & *(gray3) \, \\ *(gray3) \,  \end{ytableau}\\
				\theta_2 &{}= (3,3,3,2)\slash (3,3,1)  = \begin{ytableau} \none  & *(gray2) \, & *(gray2) \,  \\ *(gray2) \, & *(gray2)  \end{ytableau}   \\[1ex]         \theta_1 &{}= (1,1,1,1,1,1)\slash (1,1,1,1) = \begin{ytableau} *(gray1) \, \end{ytableau} 
				\end{align*}
			\end{minipage}
		\end{center}
	\end{example}
	
	\begin{definition}\label{def:consubrib}
		\begin{enumerate}[i)]
			\item For a skew diagram $\lsm$ we define its \emph{content} by
			\[c(\lsm) = \left\{j-i \mid (i,j) \in D(\lsm) \right\} \subset \Z  \,.\]
			\item We say a ribbon $R'$ is contained in another ribbon $R$, if there exist a $t \in \Z$ with
			\begin{align*}
			\big\{ (i+t,j+t) \mid (i,j) \in D(R') \big\} \subset D(R)\,.
			\end{align*}
			In this case we also call $R'$ a \emph{subribbon} of $R$. 
			\item Let $\Theta=(\theta_1,\dots,\theta_n)$ be an outside decomposition of an edge-connected skew diagram $\lsm$.  As in \cite{HG} we note that the $\theta_i$ nest correctly, so there exists a ribbon $R$, which contains all $\theta_1,\dots,\theta_n$ and which satisfies $c(R)=c(\lsm)$.
			This ribbon is unique up to diagonal translation and we will denote by $R_\Theta$ the one which is furthest left. (This is well-defined, since all coordinates in a diagram are positive.)
			\item Since $R_\Theta$  contains all $\theta_1,\dots,\theta_n$, we can define for $1 \leq i,j \leq n$ the subribbons $R_\Theta(i,j)$ of $R_\Theta$ by the property $c(R_\Theta(i,j)) = [\min c(\theta_i), \max c(\theta_j)]$ if $\min c(\theta_i)\leq \max c(\theta_i)$. In the case $\min c(\theta_i)=\max c(\theta_j)+1$ we set $R_\Theta(i,j)=\emptyset$ and in the cases $\min c(\theta_i)>\max c(\theta_j)+1$ the $R_\Theta(i,j)$ are undefined.
		\end{enumerate}
	\end{definition}
	
	\begin{example}    We again use the same skew diagram $\lsm = (4,3,3,2,1)\slash (1)$  and outside decomposition $\Theta= (\theta_1,\theta_2,\theta_3,\theta_4)$ as in Example \ref{ex:outdec}. In this case we have $c(\lsm) = \{-4,-3,\dots,2,3\}$ and
		{
			\belowdisplayskip=1em
			\belowdisplayshortskip=1em
			\[ R_\Theta = (6,6,6,5,3)\slash (6,6,4,2) =  \begin{ytableau} \none & \none &\none &\none & \, & \,\\ \none &\none & \, & \,& \,\\
			\, & \,& \,
			\end{ytableau}\,,\]
		}
		which contains the ribbons $\theta_1,\theta_2,\theta_3,\theta_4$ in the following way
		{
			\abovedisplayshortskip=1em\belowdisplayshortskip=1em
			\abovedisplayskip=1em\belowdisplayskip=1em
			\begin{equation*}
			\begin{ytableau} \none & \none &\none &\none & \, & \,\\ \none &\none & \, & \,& \,\\
			*(gray1)  & \,& \,
			\end{ytableau}\,,\quad
			\begin{ytableau} \none & \none &\none &\none & \, & \,\\ \none &\none &     *(gray2) \, &     *(gray2) \,& \,\\
			\, &    *(gray2)\, &     *(gray2)\, 
			\end{ytableau}\,,\quad
			\begin{ytableau} \none & \none &\none &\none &    *(gray3)  \, &     *(gray3) \,\\ \none &\none &     *(gray3) \, &     *(gray3) \,&     *(gray3) \,\\
			\, & \,&     *(gray3) \,
			\end{ytableau}\,,\quad
			\begin{ytableau} \none & \none &\none &\none & \, & \,\\ \none &\none & \, & \,&     *(gray4) \,\\
			\, & \,& \,
			\end{ytableau}\,.
			\end{equation*}}
		The subribbons $R_{\Theta}(i,j)$ can be read off from here, by choosing the starting box of $\theta_i$ and the ending box of $\theta_j$ in $R_{\Theta}$. We obtain
		{
			\abovedisplayshortskip=1em\belowdisplayshortskip=1em
			\abovedisplayskip=1em\belowdisplayskip=1em
			\begin{alignat*}{8}
			R_{\Theta}(1,1) &{}= \begin{ytableau} \, \end{ytableau}  =  \theta_1 \,, \quad &
			R_{\Theta}(1,2) &{}=  \begin{ytableau}  \none &\none & \, & \,\\\, & \,& \, \end{ytableau} \,\,,  \quad \quad  &
			R_{\Theta}(1,3) &{}=  R_{\Theta} \,, &
			R_{\Theta}(1,4) &{}=  \begin{ytableau}  \none &\none & \, & \,& \,\\\, & \,& \,\end{ytableau} \,\,,
			\\[1ex]
			R_{\Theta}(2,1) &{}=  \emptyset \,, \quad &
			R_{\Theta}(2,2) &{}= \begin{ytableau}  \none & \, & \,\\ \,& \, \end{ytableau} =  \theta_2 \,,&
			R_{\Theta}(2,3) &{}=  \begin{ytableau}\none & \none &\none & \,& \,\\ \none & \, & \,& \,\\ \,& \, \end{ytableau}\,\,, \quad &
			R_{\Theta}(2,4) &{}=  \begin{ytableau}  \none & \, & \,& \,\\\,& \,\end{ytableau}\,\,,
			\\[1ex]
			R_{\Theta}(3,1) &{}=  {\text{undef.}} \,, &
			R_{\Theta}(3,2) &{}=  \begin{ytableau}   \, & \,\\\,    \end{ytableau}\,\,,&
			R_{\Theta}(3,3) &{}= \begin{ytableau} \none &\none & \,& \,\\  \, & \,& \,\\ \,\end{ytableau}  = \theta_3 \,, \quad &
			R_{\Theta}(3,4) &{}=  \begin{ytableau} \, & \,& \,\\\,\end{ytableau}\,\,,
			\\[1ex]
			R_{\Theta}(4,1) &{}=  {\text{undef.}} \,,&
			R_{\Theta}(4,2) &{}=  \text{undef.} \,,&
			R_{\Theta}(4,3) &{}= \text{undef.} \,,&
			R_{\Theta}(4,4)  &{}= \begin{ytableau} \,\end{ytableau}  =  \theta_4\,.
			\end{alignat*}}
	\end{example}

	\begin{definition}
		\begin{enumerate}[i)]
			\item We denote by $T^{\rm{diag}}(\lsm)$ the set of all Young diagram $\kk=(k_{i,j})\in T(\lsm)$, which have constant entries on the diagonals, i.e. $k_{i,j} = k_{i^{\prime},j^{\prime}}$ whenever $j-i=j^{\prime}-i^{\prime}$.
			\item For an outside decomposition $\Theta=(\theta_1,\dots,\theta_n)$  of $\lsm$ and $\kk \in T^{\rm{diag}}(\lsm)$ we define the Young tableau $R^\kk_\Theta(i,j) = (r_{i,j})  \in T(R_\Theta(i,j))$  by $r_{i,j} = k_{i',j'}$ for some $(i',j')\in D(\lsm)$ with $j'-i' = i -j$. 
			(Notice that $(i,j)$ is not necessarily an element in $D(\lsm)$, but there has to be at least one $(i',j') \in D(\lsm)$ with the same content.)
		\end{enumerate}
	\end{definition}

	\begin{example} Again in the case $\lsm = (4,3,3,2,1)\slash (1)$ with the outside decomposition $\Theta= (\theta_1,\theta_2,\theta_3,\theta_4)$  of $\lsm$ as in Example \ref{ex:outdec}, we could choose 
		\ytableausetup{centertableaux, boxsize=1.1em}
		\begin{align*}
		\kk ={\footnotesize\begin{ytableau}
			\none & 2 & 1  & 5 \\
			6  & 3 &2\\
			4 & 6  & 3  \\
			7 & 4  \\
			3  \\
			\end{ytableau}}\, \in T^{\rm{diag}}(\lsm) \, .
		\end{align*}
		This would give, for example 
		\begin{align*}
		R^{\kk}_{\Theta}(1,2) =  
		{\footnotesize\begin{ytableau}  \none &\none & 6 & 3\\
			3 & 7 & 4
			\end{ytableau}}\,\quad \text{or} \quad 
		R^{\kk}_{\Theta}(3,3) = {\footnotesize\begin{ytableau} \none &\none & 1& 5\\  6 & 3& 2\\
			4
			\end{ytableau}}\,\,.
		\end{align*}
	\end{example}
	
	\begin{theorem} \label{thm:jacobitrudi} For an edge-connected skew diagram $\lsm$ with outside decomposition $\Theta=(\theta_1,\dots,\theta_n)$ and $\kk \in T^{\rm{diag}}(\lsm)$, we have for all $M \in \Z_{>0}$ 
		\begin{equation}
		\sfm(\kk) = \det\left( \sfm(R^\kk_\Theta(i,j))  \right)_{1 \leq i,j \leq n}\,,
		\end{equation}
		where we set $\sfm(R^\kk_\Theta(i,j)) =0$ if $R_\Theta(i,j)$ is undefined.
	\end{theorem}
	\begin{proof}
		The proof is a variation of the proof of Theorem 3.1 in \cite{HG} for Schur functions, which we will now describe roughly and modify for our purpose.
		
		In  \cite{HG} the ribbons are called strips and the authors define for a outside decomposition $(\theta_1,\dots,\theta_n)$ a strip $\theta_i \# \theta_j$ (\cite[p. 465]{HG}). Since we are just considering the case of edge-connected skew diagrams $\lsm$, we do not need to consider the ``null strips'' used in \cite{HG}. In particular, Case III of the definition of $\theta_i \# \theta_j$ never appears. The ``same shape'' mentioned in the Case I is exactly given by our ribbon $R_\Theta$ and from the definition of the Cases I \& II one can see that we have $\theta_i \# \theta_j = R_\Theta(i,j)$.
		
		Following the proof in \cite{HG} for a given outside decomposition $(\theta_1,\dots,\theta_n)$ we construct an $n$-tuple of lattice paths which are in a $1:1$ correspondence with semi-standard Young tableaux of shape $\lsm$. The $i$-th path begins at $P_i$ and ends at $Q_i$ for $1 \leq i \leq n$. If $\theta_i$ has a starting box on the left perimeter in box $(s,t) \in D(\lsm)$ of the diagram we set $P_i = (t-s, 1)$ and if it has a starting box on the bottom perimeter  in box $(s,t)$ (but not the left) we set $P_i=(t-s,M-1)$. If the ribbon $\theta_i$ has ending box on the top perimeter in box $(u,v)$ we set $Q_i = (v-u+1,M-1)$ and if it has ending box on the right perimeter in box $(u,v)$ (but not the top) we set $Q_i = (v-u+1,1)$.
		
		For each semi-standard Young tableaux of shape $\lsm$ we obtain a unique $n$-tuple of paths from $P_i$ to $Q_i$ for $i=1,\dots,n$, which are allowed to have horizontal or down-diagonal steps to the right, and vertical steps up and down.  If a box in the Young tableau containing $j$ at coordinates $(a,b)$ in the diagram is approached from the left in the ribbon $\theta_i$, we put a horizontal step from $(b-a,j)$ to $(b-a+1,j)$. If a box containing $j$ in the Young tableaux at coordinates $(a,b)$ is approached from below in $\theta_i$, we put a diagonal step from $(b-a,j+1)$ to $(b-a+1,j)$.
		
		For example, in the case $\lsm =(4,3,3,2,1) \slash (1)$ we show in Figure \ref{fig:pathsystem} how a path system on the left corresponds to the semi-standard Young tableau on the right, with the outside decomposition $(\theta_1,\theta_2,\theta_3,\theta_4)$ given in Example \ref{ex:outdec} and $M=8$.
		\begin{figure}[H] 
			\begin{minipage}[c]{0.4\linewidth}
				\begin{tikzpicture}[scale=0.3]
				\def\rad{2.5 pt}
				\draw [line width=0.4mm,darkgray] (-8,2) -- (-8,4) -- (-8,6) -- (-8,8) -- (-8,10) -- (-8,12) -- (-6,12) -- (-6,14);
				\draw [line width=0.4mm,darkgray] (-6,2) -- (-6,4) -- (-6,6) -- (-6,8) -- (-6,10) -- (-4,10) -- (-2,10) -- (0,8) -- (2,8) -- (2,10)--(2,12)--(2,14);
				\draw [line width=0.4mm,darkgray] (-4,2) -- (-4,4) -- (-4,6) -- (-4,8) -- (-2,8) -- (-2,6) -- (0,4) -- (0,6) -- (2,6) -- (4,6)--(6,4)--(6,6)--(6,8)--(6,10)--(8,10)--(8,12)--(8,14);
				\draw [line width=0.4mm,darkgray] (2,2) -- (2,4) -- (4,4) -- (4,2);
				\foreach \x in {-4,...,4}
				{
					\foreach \y in {1,...,7}
					{
						\fill[black] (2*\x,2*\y) circle (\rad);    
					}    
				}
				\draw ($(-8,2)-(0,0.1)$) node[below] {$P_1$};
				\fill[black] (-8,2) circle (4pt);
				\draw ($(-6,2)-(0,0.1)$) node[below] {$P_2$};
				\fill[black] (-6,2) circle (4pt);
				\draw ($(-4,2)-(0,0.1)$) node[below] {$P_3$};
				\fill[black] (-4,2) circle (4pt);
				\draw ($(2,2)-(0,0.1)$) node[below] {$P_4$};
				\fill[black] (2,2) circle (4pt);
				\draw ($(-6,14)+(0,0.1)$) node[above] {$Q_1$};
				\fill[black] (-6,14) circle (4pt);
				\draw ($(2,14)+(0,0.1)$) node[above] {$Q_2$};
				\fill[black] (2,14) circle (4pt);
				\draw ($(8,14)+(0,0.1)$) node[above] {$Q_3$};
				\fill[black] (8,14) circle (4pt);
				\draw ($(4,2)-(0,0.1)$) node[below] {$Q_4$};
				\fill[black] (4,2) circle (4pt);
				e                \end{tikzpicture}
			\end{minipage}\hfil
			\begin{minipage}[c]{0.3\linewidth}
				\begin{align*}
				\ytableausetup{centertableaux, boxsize=1.8em}
				\begin{ytableau}
				\none & 2 & 2 & 5 \\
				2 & 3 & 3\\
				4 & 4 & 4 \\
				5 & 5\\
				6
				\end{ytableau}
				\end{align*}
			\end{minipage}
			\captionsetup{width=.9\linewidth}
			\caption{An example of a lattice path system and the corresponding semi-standard Young tableau of shape $(4,3,3,2,1) \slash (1)$.}\label{fig:pathsystem}
		\end{figure}
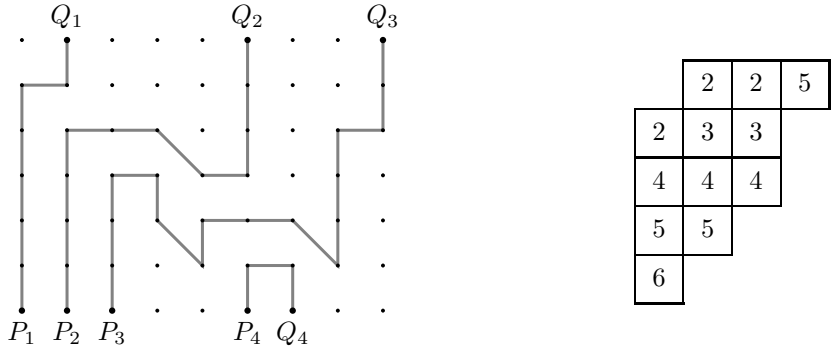
		
		A step ending in $(i,j)$ corresponds to a box in the semi-standard Young tableaux of content $i-1$ and filling $j$. The proof for the case of Schur functions in \cite{HG} then uses the usual Gessel-Viennot procedure and assigns to a horizontal or diagonal step ending in a point $(i,j)$ the weight $x_j$. To adapt this proof for the more general case of our functions $S_f^{M}$, we therefore just need to change the weight of a step ending in $(i,j)$ to be $f(j,i-1)$, which then proves the Theorem.
	\end{proof}
	
	\begin{remark} The authors would expect that there is also a version of Theorem \ref{ex:outdec} for interpolated variants of $\sfm$, which generalize interpolated Schur multiple zetas as defined in \cite{B}. This might be doable by using the construction of the $t$-lattice graph therein and a similar approach to the one described in the proof above. 
	\end{remark}
	
	In the application to Schur multiple zeta values we will usually not start with an outside decomposition, but instead, we will fix a ribbon $R_\theta$. Therefore the following corollary will be helpful later. 
	
	\begin{corollary}\label{cor:ribbonjacobi} For an edge-connected skew diagram $\lsm$, a ribbon $R$ with $c(R)=c(\lsm)$ and $\kk \in T^{\rm{diag}}(\lsm)$, the value $\sfm(\kk)$ can be written as a polynomial in $\sfm(\kk')$. Here the $\kk' \in T^{\rm{diag}}(R')$ are Young tableaux  where $R'$ are subribbons of $R$. 
	\end{corollary}
	\begin{proof}
		To prove this statement we need to construct an outside decomposition $\Theta$ with $R_{\Theta}=R$, which can be done in the following way. To construct $\theta_1$, let $(i,j) \in D(\lsm)$ the box of $R$ with smallest content $j-i$. We have $c(R)=c(\lsm)$ and therefore this box corresponds to the bottom left box of $\lsm$. Define this box to be the first box of $\theta_1$.  
		
		Since $R$ is a ribbon there is either a box on top of this box or one on the right. Assume there is a box on top: if there is also a box on top of the corresponding box in $\lsm$, choose this box to be the next box of $\theta_1$. If there is no box, $\theta_1$ is complete and it has a box starting on the bottom perimeter and ending box on the top perimeter of $\lsm$. Similarly, if there is a box on the right: If there is also a box on the right in $\lsm$ choose this box, otherwise $\theta_1$ is complete and has an ending box on the right perimeter of $\lsm$.
		
		After $\theta_1$ is complete, we can construct $\theta_2$ by choosing one of the remaining boxes with the smallest content and repeat the same construction as we did for $\theta_1$.   It is clear that in this process we will never choose a box that has already been removed: if we did, then the box of this content in the ribbon would have a neighbour below and to the left contrary to the definition of a ribbon.
		
		Once all boxes have been removed in this way, we obtain an outside decomposition $\Theta$ of $\lsm$, with $R_{\Theta}=R$. Since $R_\Theta(i,j)$ are all subribbons of $R$, the statement follows from Theorem \ref{thm:jacobitrudi}.
	\end{proof}
	
	\begin{remark}
		Another point of view to construct an outside decomposition out of a ribbon $R$, is to first `cover the plane' by copies of $R$. Since $c(\lsm) = R$ we can cover $\lsm$ completely by copies of $R$. Now the corresponding outside decomposition can be read off by considering the intersection of $\lsm$ with the copies of $R$. 
		
		Figure \ref{fig:routerdecomp} illustrates how one obtains an outside decomposition of $\lsm = (4,3,3,2,1)\slash (1)$ from the ribbon $R=(6,6,6,5,3)\slash (6,6,4,2)$, by tiling the plane in this way.
		\ytableausetup{centertableaux, boxsize=0.8em}
		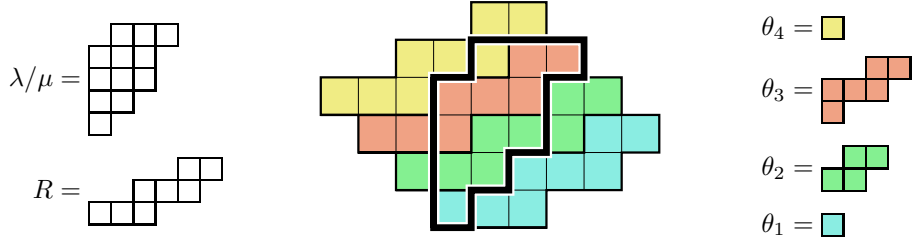
\begin{figure}[H]
			\begin{minipage}[c]{0.3\linewidth}
				\begin{align*}
				\lsm &{}={}\begin{ytableau}
				\none &\,  & \,  & \, \\
				\,  & \, & \,  \\
				\,  &\,  & \,  \\
				\,  &  \,   \\
				\,   \\
				\end{ytableau}\\[1ex]
				R&{}={}\begin{ytableau} \none & \none &\none &\none & \, & \,\\ \none &\none & \, & \,& \,\\
				\, & \,& \,
				\end{ytableau}
				\end{align*}
			\end{minipage}\hfil
			\begin{minipage}[c]{0.3\linewidth}
				\begin{tikzpicture}[scale=0.5]
				\def\rad{1.5 pt}
				\def\liff{1.5mm}        
				\def\lif{0.9mm}        
				\def\li{0.3mm}        
				\def\lis{0.1mm}        
				\foreach \t in {0,...,3}
				{
					\pgfmathsetmacro\tt{int(\t+1)}
					\path [fill,gray\tt] (3-\t,\t) -- (6-\t,0+\t) -- (6-\t,1+\t)--(8-\t,1+\t)--(8-\t,2+\t)--(9-\t,2+\t)--(9-\t,3+\t)--(7-\t,3+\t)--(7-\t,2+\t)--(5-\t,2+\t)--(5-\t,1+\t)--(3-\t,1+\t)--(3-\t,0+\t);        
					
					\draw [line width=\li,black] (3-\t,\t) -- (6-\t,0+\t) -- (6-\t,1+\t)--(8-\t,1+\t)--(8-\t,2+\t)--(9-\t,2+\t)--(9-\t,3+\t)--(7-\t,3+\t)--(7-\t,2+\t)--(5-\t,2+\t)--(5-\t,1+\t)--(3-\t,1+\t)--(3-\t,0+\t)-- (6-\t,0+\t);
					\draw [line width=\lis,black] (4-\t,0+\t) -- (4-\t,1+\t);        
					\draw [line width=\lis,black] (5-\t,0+\t) -- (5-\t,1+\t) -- (6-\t,1+\t)-- (6-\t,2+\t);        
					\draw [line width=\lis,black] (7-\t,1+\t) -- (7-\t,2+\t) -- (8-\t,2+\t)--(8-\t,3+\t);        
				}
				
				\draw [line width=\liff,white] (3,0)--(3,4)--(4,4)--(4,5)--(7,5)--(7,4)--(6,4)--(6,2)--(5,2)--(5,1)--(4,1)--(4,0)--(3,0)--(3,4);
				\draw [line width=\lif,black] (3,0)--(3,4)--(4,4)--(4,5)--(7,5)--(7,4)--(6,4)--(6,2)--(5,2)--(5,1)--(4,1)--(4,0)--(3,0)--(3,4);
				\end{tikzpicture}
			\end{minipage}
			\begin{minipage}[c]{0.3\linewidth}        
				\begin{align*}
				\theta_4 &{}={} \begin{ytableau} *(gray4) \, \end{ytableau}\\    
				\theta_3  &{}={}\begin{ytableau} \none &  \none& *(gray3) \, & *(gray3) \,  \\ *(gray3) \, & *(gray3) \,  & *(gray3) \, \\ *(gray3) \,  \end{ytableau} \\[1ex]        
				\theta_2 &{}={}\begin{ytableau} \none  & *(gray2) \, & *(gray2) \,  \\ *(gray2) \, & *(gray2)  \end{ytableau} \\[1ex]
				\theta_1 &{}={} \begin{ytableau} *(gray1) \, \end{ytableau}
				\end{align*}
			\end{minipage}
			\captionsetup{width=.9\linewidth}
			\caption{Starting with a skew diagram $\lsm = (4,3,3,2,1)\slash (1)$ and a ribbon $R=(6,6,6,5,3)\slash (6,6,4,2)$, we can cover the plane by $R$ and consider its intersections with $\lsm$. This gives an outside decomposition $\Theta=(\theta_1,\theta_2,\theta_3,\theta_4)$ with $R_\Theta=R$.}\label{fig:routerdecomp}
		\end{figure}
	\end{remark}
	
	\section{Jacobi-Trudi formula for regularized Schur multiple zeta values}\label{sec:smzv}
	
	We now recall the definition of Schur multiple zeta values, which were first defined in \cite{NPY}. For a Young tableau\footnote{The entries of the Young tableau could also be elements in $\C$ like in \cite{NPY}. We will just consider the case of integer entries.} $\kk = (k_{i,j}) \in T(\lambda \slash \mu)$  and an integer $M \in \Z_{>0}$ the truncated Schur multiple zeta value is defined by  
	\begin{equation}\label{eq:defschurmzv}
	\zeta_M(\kk) = \sum_{(m_{i,j}) \in \SSYT_M(\lambda \slash \mu)} \,\, \prod_{(i,j) \in D(\lambda \slash \mu)}  \frac{1}{m_{i,j}^{k_{i,j}}} \,. 
	\end{equation}
	The limit $M\rightarrow \infty$ of \eqref{eq:defschurmzv} exist, when $\kk$ is admissible. Being admissible means that $k_{i,j} \geq 2$ for all corners $(i,j)$, where $(i,j)\in D(\lambda\slash\mu)$ is called a corner of $\lambda \slash \mu$
	if $ (i,j+1) \not\in D(\lambda\slash\mu)$ and $ (i+1,j) \not\in D(\lambda\slash\mu)$, see \cite[Lemma~2.1]{NPY}. For an admissible $\kk$ the Schur multiple zeta value $ \zeta(\kk) $ is then defined by 
	\[ \zeta(\kk) = \lim_{M \rightarrow \infty} \zeta_M(\kk)\,.\]
	The numbers \eqref{eq:defschurmzv} generalize truncated multiple zeta and zeta-star values
	\begin{align*} 
	\zeta_M(k_1,\ldots,k_r)&{}=\sum_{0<m_1<\cdots<m_r<M} \frac{1}{m_1^{k_1}\cdots m_r^{k_r}} \,, \quad \zeta^\star_M(k_1,\ldots,k_r)&{}=\sum_{0<m_1\leq\cdots \leq m_r<M} \frac{1}{m_1^{k_1}\cdots m_r^{k_r}} \,,
	\end{align*}
	by choosing a column of length $r$ or row of length $r$ respectively. In these cases the only corner is the box containing $k_r$ and therefore the limit $M\rightarrow \infty$ exists in the case $k_r\geq 2$, which then gives the multiple zeta (star) values $\zeta(k_1,\dots,k_r)$ and $\zeta^\star(k_1,\dots,k_r)$ from \eqref{eq:mzv}.
	
	For any Young tableau $\kk$, the value $\zeta_M(\kk)$ can be written as a linear combination of truncated multiple zeta values by considering the so-called topological sorts (or rather a generalization thereof) of the poset given by the inequalities of the semi-standard Young tableau.  For example, we have
	\ytableausetup{centertableaux, boxsize=1.2em}
	\begin{align} \label{eq:schurasmzv}
	\begin{split}
	\zeta_M\left({\footnotesize
		\begin{ytableau}
		a & b \\ c
		\end{ytableau}}\right) &{}= \sum_{
		{\scriptsize\arraycolsep=1.4pt\def\arraystretch{0.8}
			\begin{array}{cccc}
			0<&m_a&\leq m_b &<M\\
			&\vsmall &  \\
			&m_c&< M
			\end{array} }} \frac{1}{m_a^a \cdot m_b^b \cdot m_c^c} = \sum_{\substack{\phantom{\text{or }} 0 < m_a < m_c < m_b < M\\ \text{or } 0 < m_a < m_c = m_b < M \\ \text{or } 0 < m_a < m_b < m_c < M \\ \text{or } 0 < m_a = m_b < m_c < M}} \frac{1}{m_a^a \cdot m_b^b \cdot m_c^c} \\
	&= \zeta_M(a,c,b)+\zeta_M(a,b+c) + \zeta_M(a,b,c)+\zeta_M(a+b,c)\,.
	\end{split}
	\end{align}
	
	The product of two multiple zeta values can be expressed by using the so called harmonic product formula, which in the lowest depth is given by $\zeta_M(a) \zeta_M(b) = \zeta_M(a,b) + \zeta_M(b,a) + \zeta_M(a+b)$. Using the notion of Schur multiple zeta values, this can be expressed even more nicely as
	\[\zeta_M\left(\footnotesize \ytableausetup{centertableaux, boxsize=1.2em}
	\begin{ytableau}
	a
	\end{ytableau} \right) \zeta_M\left({\footnotesize \ytableausetup{centertableaux, boxsize=1.2em}
		\begin{ytableau}
		b
		\end{ytableau}}\right)  = \zeta_M\left( {\footnotesize \ytableausetup{centertableaux, boxsize=1.2em}
		\begin{ytableau}
		a & b
		\end{ytableau}}\right) + \zeta_M\left({\footnotesize \ytableausetup{centertableaux, boxsize=1.2em}
		\begin{ytableau}
		b\\a
		\end{ytableau}}\right)  \,.\]
	See \cite{BY} for a precise definition of the harmonic product of Schur multiple zeta values for general Young tableaux. In \cite{IKZ} the authors introduce (harmonic) regularized\footnote{In their work Ihara-Kaneko-Zagier define two regularized versions (harmonic and shuffle) of multiple zeta values, denoted by $Z^\ast_{k_1,\dots,k_r}(T) \in \R[T]$ and $Z^\shuffle_{k_1,\dots,k_r}(T)\in \R[T]$. We just consider the harmonic version and write $\zr(k_1,\dots,k_r ; T) = Z^\ast_{k_1,\dots,k_r}(T)$.} multiple zeta values $\zr(k_1,\dots,k_r; T)\in \R[T]$, which are defined for all $k_1,\dots,k_r \in \Z_{>0}$. These coincide with the classical multiple zeta values in the case $k_r \geq 2$ and they satisfy the harmonic product formula for all indices. We generalize this notion to Schur multiple zeta values in the following way.
	
	\begin{lemma} \label{lem:schurmzvasym}
		For any Young tableau $\kk  \in T(\lambda \slash \mu)$ there exist a unique polynomial $\zr(\kk ; T)\in \R[T]$, such that 
		\begin{align*}
		\zeta_M(\kk) = \zr(\kk ; \log(M) + \gamma ) +O\left(\frac{\log^J(M)}{M}\right)
		\end{align*}
		as $M\rightarrow \infty$. Here $\gamma$ denotes the Euler–Mascheroni  constant and $J$ is an integer depending on $\kk$.
	\end{lemma}
	\begin{proof} As mentioned before, the value $\zeta_M(\kk)$ can be written as a linear combination of the truncated multiple zeta values, by using the same idea as in \eqref{eq:schurasmzv}.
		By \cite{IKZ} these satisfy the similar asymptotic formula
		\[ \zeta_M(k_1,\dots,k_r) = \zr(k_1,\dots,k_r; \log(M) + \gamma) + O\left(\frac{\log^{J'}(M)}{M}\right)\]
		for some integer $J'$, from which the result follows. 
	\end{proof}

	\begin{remark}
		Even though $\zr(\kk ; T)$ depends on $T$, we will omit $T$ from our notation and just write $\zr(\kk)$ in the following. In the literature of classical multiple zeta values people typically refer to  $\zr(k_1,\dots,k_r;0)\in \R$  as the regularized multiple zeta values.  In the following we do not specialize $T$ and always view $\zr(\kk)\in \R[T]$  as a polynomial.
	\end{remark}
	
	For any Young tableau $\kk=(k_{i,j})\in T(\lsm)$ it is easy to see that for $f(m,d) = m^{-d}$ and all $M\in \Z_{>0}$, we have
	\begin{equation}\label{eq:zetaass}
	\zeta_M(\kk) =  \sfm(\kk) \,. 
	\end{equation}
	Using Theorem \ref{thm:jacobitrudi}  in the case of Young tableaux with constant diagonal entries we therefore obtain the following.
	
	\begin{theorem} \label{thm:jacobiformzv} For an edge-connected skew diagram $\lsm$ with outside decomposition $\Theta=(\theta_1,\dots,\theta_n)$ and $\kk \in T^{\rm{diag}}(\lsm)$, we have 
		\begin{equation}
		\zr(\kk) = \det\left( \zr(R^\kk_\Theta(i,j))  \right)_{1 \leq i,j \leq n}\,,
		\end{equation}
		where we set $\zr(R^\kk_\Theta(i,j)) =0$ if $R_\Theta(i,j)$ is undefined.
	\end{theorem}
	\begin{proof}
		By \eqref{eq:zetaass} and  Theorem \ref{thm:jacobitrudi} we get for all $M\in \Z_{>0}$
		\[     \zeta_M(\kk) = \det\left( \zeta_M(R^\kk_\Theta(i,j))  \right)_{1 \leq i,j \leq n}\,.\]
		The result follows from Lemma \ref{lem:schurmzvasym} by comparing the asymptotics of both sides as $M \rightarrow \infty$. 
	\end{proof}
	Theorem \ref{thm:jacobiformzv} generalizes the Jacobi-Trudi type formulae proven in \cite{NPY}. Notice that even when $\kk$ is admissible, the $R^\kk_\Theta(i,j)$ are not necessarily all admissible and therefore in the matrix regularization is necessary.  The terms involving $T$ will nevertheless cancel out in the determinant in these cases. Clearly, we also have an analogue of Corollary \ref{cor:ribbonjacobi}, which will be the more interesting point of view in our application below. 
	\begin{corollary}\label{cor:ribbonjacobimzv}
		For an edge-connected skew diagram $\lsm$ and a ribbon $R$ with $c(R)=c(\lsm)$ and $\kk \in T^{\rm{diag}}(\lsm)$, the $\zr(\kk)$ can be written as a polynomial in $\zr(\kk'; T)$. Here the $\kk' \in T^{\rm{diag}}(R')$ are Young tableaux with the $R'$ being subribbons of $R$.
	\end{corollary}

	\section{Checkerboard style Schur multiple zeta values }
	In \cite{BY} the authors evaluated special type of Schur multiple zeta values which have alternating entries in $a,b \geq 1$. For example the Schur multiple zeta value
	\begin{align*}
	\zr\left(\ {\footnotesize
		\begin{ytableau}
		\none & \none & \none & a & *(gray)b\\
		\none & *(gray)b & a & *(gray)b \\
		*(gray)b & a & *(gray)b
		\end{ytableau}} \ \right) 
	\end{align*}
	is of this type.  Here the coloring is just for aesthetic purposes. We refer to these kind of Schur multiple zeta values as \emph{Checkerboard style Schur multiple zeta values}. 
	We first recall some notation introduced in \cite{BY}. For $n\geq 0$ define the following four type of Checkerboard style Schur multiple zeta values
	\begin{alignat*}{4} 
	\ytableausetup{mathmode,boxsize=1.1em,aligntableaux=center} 
	A_{a,b}(n)
	&\coloneqq
	\zr
	\left(
	\ {\footnotesize
		\begin{ytableau}
		\none & \none    & \none  & a        \\
		\none & \none    & \adots & *(gray)b \\
		\none & a        & \adots \\
		a     & *(gray)b  
		\end{ytableau}}
	\ \right)
	\,,\qquad &
	B_{a,b}(n)
	&\coloneqq
	\zr\left(
	\ {\footnotesize
		\begin{ytableau}
		\none  & \none & a & *(gray)b \\
		\none  & \adots& *(gray)b \\
		a & \adots \\
		*(gray)b
		\end{ytableau}}
	\
	\right)\,,\\
	S_{a,b}(n)&\coloneqq
	\zr
	\left(
	\ {\footnotesize
		\begin{ytableau}
		\none    & \none  & a        \\
		\none    & \adots & *(gray)b \\
		a        & \adots \\
		*(gray)b 
		\end{ytableau}}
	\ \right) \,,\qquad&
	S_{a,b}^{\star}(n)&\coloneqq
	\zr
	\left(
	\ {\footnotesize
		\begin{ytableau}
		\none & \none  & a      & *(gray)b \\
		\none & \adots & \adots \\
		a     & *(gray)b 
		\end{ytableau}}
	\ \right)\,.
	\end{alignat*}
	Here $n$ denotes the number of pairs of $a$ and $b$ and in the case $n=0$ we interpret above Schur multiple zeta values as  $A_{a,b}(0)=\zr(a)$, $B_{a,b}(0)=\zr(b)$ and $S_{a,b}(0)=S^\star_{a,b}(0)=1$. 
	In the case when $(a,b)=(1,3)$ the above checkerboard style Schur multiple zeta values can be evaluated explicitly. 
	\begin{theorem}[{\cite[Theorem 3.4 \& 3.5]{BY}}]
		\label{thm:SSstar}
		For $n\ge 1$ we have  
		\begin{align*}
		\ytableausetup{mathmode,boxsize=1.2em,aligntableaux=center} 
		S_{1,3}(n)
		&{}=\zeta
		\left(\:\,\:{\footnotesize
			\ 
			\begin{ytableau}
			\none    & \none     & 1        \\
			\none    & \adots    & *(gray)3 \\
			1        & \adots \\
			*(gray)3 
			\end{ytableau}\:\:}
		\ \right) 
		= \frac{1}{4^n} \zeta^\star(\{4\}^n) \,,\\
		S^\star_{1,3}(n)
		&=\zeta
		\left({\footnotesize
			\ 
			\begin{ytableau}
			\none        & \none  & 1      & *(gray)3 \\
			\none        & \adots & \adots \\
			1 & *(gray)3 
			\end{ytableau}}
		\ \right)
		= \sum_{k=0}^n \frac{1}{4^k} \zeta^\star(\{4\}^k) \zeta(\{4\}^{n-k})
		\,,\\
		A_{1,3}(n)
		&{}=\zeta
		\left(
		\ {\footnotesize
			\begin{ytableau}
			\none & \none    & \none  & 1        \\
			\none & \none    & \adots & *(gray)3 \\
			\none & 1        & \adots \\
			1     & *(gray)3  
			\end{ytableau}}
		\ \right)
		=\frac{2}{4^n} \zeta(4n+1)\,,\\
		B_{1,3}(n)
		&=\zeta
		\left(
		\ {\footnotesize
			\begin{ytableau}
			\none   & \none  & 1        & *(gray)3 \\
			\none   & \adots & *(gray)3 \\
			1       & \adots \\
			*(gray)3
			\end{ytableau}}
		\ \right) 
		=\frac{1}{4^{n}} \zeta(4n+3)\,.
		\end{align*} 
	\end{theorem}
	These formulae can be seen as analogues of the well-known $1$-$3$ formula for multiple zeta values:
	\begin{equation}\label{eq:31formula}
	\zeta
	\left(\ {\footnotesize
		\begin{ytableau}
		1        \\
		*(gray)3 \\
		\svdots    \\
		1        \\
		*(gray)3 
		\end{ytableau}}
	\ \right) = \zeta(\{1,3\}^n) = \frac{2\pi^{4n}}{(4n+2)!} = \frac{1}{4^n} \zeta(\{4\}^n)\,.
	\end{equation}
	In particular we see, that all the above examples are elements of $\Q[\pi^4,\zeta(3),\zeta(5),\dots]$, using the well-known facts that $\zeta^\star(\{4\}^k), \zeta(\{4\}^{n}) \in \Q[\zeta(4)] = \Q[\pi^4]$ (see for example \cite[Section 6]{HI}). We now give some other explicit evaluations for non-admissible $1$-$3$ indices.
	\begin{proposition} \label{prop:regformulas}
		For $n \geq 0$ we have
		\begin{align*}
		\zr(\{1,3\}^n,1) ={} & \zeta^*(\{1,3\}^n)T + \frac{1}{2^{2n-1}} \sum_{j=1}^n (-1)^j \zeta(4j+1) \zeta(\{4\}^{n-j})\,,\\
		\zr(\{3,1\}^n) ={} &  \zeta(\{3,1\}^{n-1},3)T +(-1)^n \sum_{k=0}^n \frac{1}{4^k} \zeta^\star(\{4\}^k) \zeta(\{4\}^{n-k})\\
		&+ \frac{1}{2^{2n-3}}\sum_{\substack{1 \leq j \leq n-1\\0 \leq k \leq n-1-j}} (-1)^{j+k} \zeta(4j+1) \zeta(4k+3) \zeta(\{4\}^{n-j-1-k}) \,.
		\end{align*}
	\end{proposition}
	\begin{proof}
		By the harmonic product (\cite[Lemma 2.2]{BY}) we have that for all $M\in \Z_{>0}$,
		\begin{align*}
		\zeta_M(\{1,3\}^n,1)  = 
		\zeta_M
		\left(
		\ {\footnotesize
			\begin{ytableau}
			1        \\
			*(gray)3 \\
			\svdots  \\
			1        \\
			*(gray)3 \\
			1
			\end{ytableau}}
		\ \right) 
		&{}= \zeta_M
		\left(
		{\footnotesize
			\begin{ytableau}
			1     
			\end{ytableau}}
		\right) 
		\zeta_M
		\left(
		\ {\footnotesize
			\begin{ytableau}
			1        \\
			*(gray)3 \\
			\svdots  \\
			1        \\
			*(gray)3 
			\end{ytableau}}
		\ \right) 
		-
		\zeta_M
		\left(
		\ {\footnotesize
			\begin{ytableau}
			\none & 1        \\
			\none & *(gray)3 \\
			\none & \svdots  \\
			\none & 1        \\
			1     & *(gray)3
			\end{ytableau}}
		\ \right) \,.
		\end{align*}
		The first equation in the Proposition follows from this together with Corollary 3.6 in \cite{BY}, which gives an explicit formula for the Schur multiple zeta value on the right, and that  $\zr(1)=T$.
		
		Similarly we can use the harmonic product formula to get that
		\[ \zr(\{3,1\}^n) =  \zeta(\{3,1\}^{n-1},3)T + \sum_{j=1}^{n-1} (-1)^{j} A_{1,3}(j) \zeta(\{3,1\}^{n-j-1},3) + (-1)^{n} S^\star_{1,3}(n) \,. \]
		Using the explicit formulae in \cite{BY} for $A_{1,3}(j)$ and $S_{1,3}(n)$ above, together with 
		\begin{align}\label{eq:mzv313}
		\zeta(3,\{1,3\}^n)
		= \sum_{k=0}^{n} \left(-\frac{1}{4}\right)^k \zeta(4k+3) \zeta(\{1,3\}^{n-k}) \,,
		\end{align}
		which is part of \cite[Corollary 3.6]{BY}, we get the desired result.
	\end{proof}

	\begin{theorem}\label{thm:13zeta}
		Every (regularized) Checkerboard style Schur multiple zeta value with alternating entries of $1$ and $3$ is an element in $\Q[\pi^4,\zeta(3),\zeta(5),\dots][T]$.
	\end{theorem}
	\begin{proof}
		By Corollary \ref{cor:ribbonjacobimzv} every $1$-$3$ Checkerboard style Schur multiple zeta value can be written as a polynomial in $\zeta(\{1,3\}^n),\zeta(3,\{1,3\}^n),\zr(\{3,1\}^n)$ and, $\zr(1,\{3,1\}^n)$, by choosing a long enough column as the ribbon $R$. The statement then follows from \eqref{eq:31formula}, \eqref{eq:mzv313} and Proposition \ref{prop:regformulas}.
	\end{proof}
	Notice that in some cases it is better to choose a different ribbon $R$ instead of a long column, since the stairs $S_{1,3}$, $S^\star_{1,3}$, $A_{1,3}$ and $B_{1,3}$ have nicer evaluations than the multiple zetas $\zeta(\{1,3\}^n),\zeta(3,\{1,3\}^n),\zr(\{3,1\}^n)$ and $\zr(1,\{3,1\}^n)$. In particular by choosing $\ytableausetup{centertableaux, boxsize=0.5em} R = {\footnotesize 
		\begin{ytableau}
		\none & \,& \,\\
		\, & \,\\
		\,
		\end{ytableau}}$ we obtain the earlier formula \eqref{eq:3x3square13} for the $3 \times 3$ square.
	\begin{proposition}\label{prop:fstair} Fix one $F\in \{S, S^\star, A,B \}$ and let $\kk$ be an Checkerboard style Young tableau with alternating entries of $a, b \geq 1$. If $\kk$ can be tessellated purely by $F_{a,b}$ stairs, then $\zr(\kk)\in \Q[F_{a,b}(n) \mid n \geq 0 ]$. 
	\end{proposition}
	\begin{proof}
		Again we use Corollary \ref{cor:ribbonjacobimzv} and choose a particular ribbon $R$. In this case, we choose the ribbon $R$ to be the one with the same shape as $F$ and such that $c(R) = c(\lsm)$.  We will call this an $F$-type stair in the following. This choice is possible since $\kk$ is tessellated purely by $F$-type stairs and therefore the box with the smallest content is the starting box of a $F$-type stair and the box with the largest content is the ending box of a $F$-type stair. The construction of the outside decomposition $\Theta$ in the proof of Corollary \ref{cor:ribbonjacobi} then assures that all the subribbons $R_\Theta(i,j)$ are also $F$-type stairs. The entries in the matrix of Corollary \ref{cor:ribbonjacobimzv} are therefore all of the form $F_{a,b}(n)$ for some $n\geq 0$ which then gives $\zr(\kk)\in \Q[F_{a,b}(n) \mid n \geq 0 ]$.  \end{proof}

	\begin{corollary} \label{cor:13tessel}Let $\kk$ be an admissible Checkerboard style Young tableau with alternating entries of $1$ and $3$.
		\begin{enumerate}[i)]
			\item If $\kk$ can be tessellated purely by $S_{1,3}$ stairs, then $\zeta(\kk) \in \Q[\pi^4]$.
			\item If $\kk$ can be tessellated purely by $S^\star_{1,3}$ stairs, then $\zeta(\kk) \in \Q[\pi^4]$.
			\item  If $\kk$ can be tessellated purely by $A_{1,3}$ stairs,  then $\zeta(\kk) \in \Q[\zeta(4n+1) \mid n\geq 1]$.
			\item If $\kk$ can be tessellated purely by $B_{1,3}$ stairs,  then $\zeta(\kk) \in \Q[\zeta(4n+3) \mid n\geq 0]$.
		\end{enumerate}
	\end{corollary}
	\begin{proof}
		This is a direct consequence of Proposition \ref{prop:fstair} and the explicit formulae for the Schur multiple zeta values $S_{1,3}$, $S^\star_{1,3}$, $A_{1,3}$ and $B_{1,3}$ given in Theorem \ref{thm:SSstar}.
	\end{proof}
	
	Proposition \ref{prop:fstair} and Corollary \ref{cor:ribbonjacobimzv} give a more direct proof for Theorem 4.4 in \cite{BY}. There the authors used the Jacobi-Trudi formula for a column $R$ instead of a stair $R$, to show that certain thick stairs can be written as polynomials in the stairs $B_{a,b}$. We want to emphasize that the general Jacobi-Trudi determinant in  Corollary \ref{cor:ribbonjacobimzv}  is much more elegant for this purpose since no complicated matrix manipulations are necessary anymore.
	
	\begin{example}
		\begin{enumerate}[i)]
			\item[i), ii)\hspace{-1em}] \hspace{1em} We note that tableaux of the form \( \lsm \) where \( \lambda_i - \mu_i \in 2\Z \) and \( \lambda_i - \lambda_{i-1} \) odd can be tessellated purely by $S^\star$-type stairs.  The transpose can be tessellated purely by $S$-type stairs.  Hence we get, for example
			\begin{equation*}
			\ytableausetup{mathmode,boxsize=1.2em,aligntableaux=center} 
			\zeta\left(\ {\footnotesize
				\begin{ytableau}
				\none&\none&\none&\none&\none&\none&1&*(gray)3&1&*(gray)3\\
				\none&\none&\none&1&*(gray)3&1&*(gray)3&1&*(gray)3\\
				\none&\none&1&*(gray)3&1&*(gray)3\\
				\none&1&*(gray)3&1&*(gray)3\\
				1&*(gray)3
				\end{ytableau}
			}\ \right),
			\zeta\left(\ {\footnotesize
				\begin{ytableau}
				\none&\none&\none&1\\
				\none&\none&1&*(gray)3\\
				\none&1&*(gray)3&1\\
				\none&*(gray)3&1&*(gray)3\\
				\none&1&*(gray)3\\
				1&*(gray)3\\
				*(gray)3
				\end{ytableau}
			}\ \right) \in \Q[\pi^4] \,.
			\end{equation*}
			\item[iii)] We note that admissible tableaux of the form \( \lambda \slash (2n, 2n-1, \ldots, 2, 1) \) can be tessellated purely by $A$-type stairs.  Hence we get, for example
			\begin{equation*}
			\ytableausetup{mathmode,boxsize=1.2em,aligntableaux=center} 
			\zeta\left(\ {\footnotesize
				\begin{ytableau}
				\none&\none&\none&\none&\none&1\\
				\none&\none&\none&\none&1&*(gray)3\\
				\none&\none&\none&1&*(gray)3&1\\
				\none&\none&1&*(gray)3&1&*(gray)3\\
				\none&1&*(gray)3&1&*(gray)3\\
				1&*(gray)3
				\end{ytableau}
			}\ \right) \in \Q[\zeta(4n + 1) \mid n \geq 1] \,.
			\end{equation*}
			\item[iv)] We note that tableaux of the form \( (2n-1, 2n-2, \ldots, 2, 1) \slash \mu \), where \( \mu \subset (2n-3, 2n-4, \ldots, 2, 1) \) can be tessellated purely by $B$-type stairs.  Hence we get, for example
			\begin{equation*}
			\ytableausetup{mathmode,boxsize=1.2em,aligntableaux=center} 
			\zeta\left(\ {\footnotesize
				\begin{ytableau}
				\none&\none&\none&\none&*(gray)3&1&*(gray)3\\
				\none&\none&\none&*(gray)3&1&*(gray)3\\
				\none&1&*(gray)3&1&*(gray)3\\
				\none&*(gray)3&1&*(gray)3\\
				*(gray)3&1&*(gray)3\\
				1&*(gray)3\\
				*(gray)3
				\end{ytableau}
			}\ \right) \in \Q[\zeta(4n + 3) \mid n \geq 0] \,.
			\end{equation*}
		\end{enumerate}
	\end{example}
	
	We can also consider Checkerboard style Schur multiple zeta values with alternating $1$ and $2$ entries.  In \cite{BY}, some results were stated about the \( A, B, S, S^\star \) stairs with these entries, from which we obtain further results using Proposition \ref{prop:fstair}.
	
	\begin{lemma} \label{lem:a12} For $n\geq 1$ we have
		\begin{align*}
		\label{for:A13}
		\ytableausetup{mathmode,boxsize=1.2em,aligntableaux=center} 
		A_{1,2}(n)
		=\zeta
		\left(
		\ {\footnotesize
			\begin{ytableau}
			\none & \none    & \none  & 1        \\
			\none & \none    & \adots & *(gray)2 \\
			\none & 1        & \adots \\
			1     & *(gray)2  
			\end{ytableau}}
		\ \right)
		=3 \zeta(3n+1)\,.
		\end{align*}
	\end{lemma}
	\begin{proof}
		This statement was stated in \cite{BY} without proof, which we will give now.
		By \cite[Lemma 2.3]{BY} we have for $n\geq 1$
		\begin{align*}
		A_{1,2}(n)&{}=(-1)^{n-1}L_{1,2}(n)-\sum^{n-1}_{k=1}(-1)^{n-k}A_{1,2}(k)\,\zeta(\{1,2\}^{n-k})\,,
		\end{align*}
		where for $n\geq 1$ the $L_{1,2}(n)$ are given by 
		\begin{align*}
		L_{1,2}(n)=
		\zeta
		\left(
		\ {\footnotesize
			\begin{ytableau}
			\none & 1       \\
			\none & *(gray)2 \\
			\none & \svdots  \\
			\none & 1        \\
			1    & *(gray)2
			\end{ytableau}}
		\ \right) \,.
		\end{align*}
		These can be written in terms of multiple zeta values by using the iterated integral expression in Equation 6.3 of \cite{NPY}.  From this we obtain the following expression
		\begin{align*}
		L_{1,2}(n) &{}= 3\sum_{\substack{l+m = n-1\\l,m \geq 0}} \zeta(\{1,2\}^l,1,1,2,\{1,2\}^m) = 3\sum_{\substack{l+m = n-1\\l,m \geq 0}} \zeta(\{3\}^m,4,\{3\}^l) \,.
		\end{align*}
		Here the last equality follows from the duality formula for multiple zeta values. Clearly we have  $A_{1,2}(1) = 3 \zeta(4)$ and by using induction on $n$, and the duality $\zeta(\{1,2\}^{n-k}) = \zeta(\{3\}^{n-k})$, we must only show
		\begin{align*}
		3\zeta(3n+1) = (-1)^{n-1}3\sum_{\substack{l+m = n-1\\l,m \geq 0}} \zeta(\{3\}^m,4,\{3\}^l)- 3\sum^{n-1}_{k=1}(-1)^{n-k}\zeta(3k+1) \,\zeta(\{3\}^{n-k})\,.
		\end{align*}
		We see this follows directly by the harmonic product formula for multiple zeta values, since terms in the second summation telescope. 
	\end{proof}
	\begin{corollary}\label{cor:12tessel}
		Let $\kk$ be an admissible Checkerboard style Young tableau with alternating entries of $1$ and $2$.
		\begin{enumerate}[i)]
			\item If $\kk$ can be tessellated purely by $S_{1,2}$ stairs, then $\zeta(\kk) \in \Q[\zeta(3n) \mid n\geq 1]$.
			\item If $\kk$ can be tessellated purely by $S^\star_{1,2}$ stairs, then $\zeta(\kk) \in \Q[\zeta(3n) \mid n\geq 1\text{ odd}]$.
			\item  If $\kk$ can be tessellated purely by $A_{1,2}$ stairs,  then $\zeta(\kk) \in \Q[\zeta(3n+1) \mid n\geq 1]$.
		\end{enumerate}
	\end{corollary}
	\begin{proof} Using the duality $\zeta(\{1,2\}^n)=\zeta(\{3\}^n)$ together \cite[Lemma 3.3]{BY} we obtain $S_{1,2}(n) = \zeta^\star(\{3\}^n) \in \Q[\zeta(3n) \mid n\geq 1]$. Then i) follows by using Proposition \ref{prop:fstair} with $F=S$.
		
		For ii) we first notice that $S^\star_{1,2}(n) = \zeta^\star(\{1,2\}^n)$, which follows from \cite[Lemma 3.3]{BY} together with \cite[Equation 38]{HI}. Then using the formula
		\[ \zeta^\star(\{1,2\}^n) = \sum_{i_1 + 3i_3 + 5i_5 + \dots = n} \frac{2^{i_1+i_3+i_5+\dots} \zeta(3)^{i_1} \zeta(9)^{i_3} \zeta(15)^{i_5} \dots }{1^{i_1} i_1! 3^{i_3} i_3! 5^{i_5} i_5! \dots} \,, \]
		which can be found in  \cite[Equation 39]{HI}, we obtain the desired result by using Proposition \ref{prop:fstair} with $F=S^\star$. 
		
		The last statement iii) is a consequence of Lemma \ref{lem:a12} and Proposition \ref{prop:fstair} with $F=A$.
	\end{proof}

	We end this note by answering another question posed in \cite{BY} Section 5. There it was observed that the product $B_{1,3}(n-1) \cdot A_{1,3}(n)$ minus their `gluings' is always a rational multiple of $\pi^{8n}$.  For example in the case $n=2$, we have
	\begin{align*}
	\zeta\left(\ {\footnotesize \begin{ytableau}
		1 & *(gray)3\\
		*(gray)3
		\end{ytableau}}\ \right) \cdot 
	\zeta\left(\ {\footnotesize \begin{ytableau}
		\none & \none & 1 \\
		\none & 1 & *(gray)3 \\
		1 & *(gray)3 
		\end{ytableau}}\ \right)  
	- \zeta\left(\ {\footnotesize    \begin{ytableau}
		1 & *(gray)3 & 1 \\
		*(gray)3 & 1 & *(gray)3 \\
		1 & *(gray)3 
		\end{ytableau}}\ \right)&= 1074502\, \zeta(\{1,3\}^4)\,.
	\end{align*}
	We denote the gluing of a stair of $B_{1,3}(n-1)$ on top of $A_{1,3}(n)$ for $n\geq 1$ by the following Schur multiple zeta value of a Young tableau with shape $ (n+1, n+1, n, \dots, 3, 2) \slash (n-2, n-3,\dots, 2,1)$:
	\begin{align*}
	G_{1,3}(n)
	= {}
	%stupid trick to get the brackets around zeta with overbrace to look nice
	% hidden version to ensure correct line height
	\mathrlap{\phantom{\zeta 
			{ 
				\overbrace{\footnotesize
					\begin{ytableau}
					\none & \none &  1   &  *(gray)3  & 1        \\
					\none &  \adots   &  *(gray)3  & 1     &  *(gray)3    \\
					1  & \adots    & \adots &  *(gray)3  \\
					*(gray)3 & 1      & \adots \\
					1     & *(gray)3 
					\end{ytableau}}^{n+1}} }}
	\zeta\left(
	\ {
		% smash to kill the vertical height, so overbrace can go outside the auto-sized brackets
		\smash{\overbrace{\footnotesize
				\begin{ytableau}
				\none & \none &  1   &  *(gray)3  & 1        \\
				\none &  \adots   &  *(gray)3  & 1     &  *(gray)3    \\
				1  & \adots    & \adots &  *(gray)3  \\
				*(gray)3 & 1      & \adots \\
				1     & *(gray)3 
				\end{ytableau}}^{n+1}}        
		% hidden version of the tableau to force brackets to be the right size
		\mathclap{\phantom{\footnotesize\begin{ytableau}
				\none & \none &  1   &  *(gray)3  & 1        \\
				\none &  \adots   &  *(gray)3  & 1     &  *(gray)3    \\
				1  & \adots    & \adots &  *(gray)3  \\
				*(gray)3 & 1      & \adots \\
				1     & *(gray)3 
				\end{ytableau}}}
	}
	\ \right)   \,.
	\end{align*}
	The observation in \cite{BY} is then a direct consequence of our generalized Jacobi-Trudi formula.  We first give a Lemma containing a fun evaluation of \( S_{1,3}^\star(n) \) in terms of Bernoulli polynomials; the result on \( G_{1,3}(n) \) follows afterwards.
	\begin{lemma}\label{lem:stairstareval}
		For \( n \geq 1 \), the following evaluation holds
		\[
		S^\star_{1,3}(n) = 2 i B_{4 n+1}\left(\frac{1 - i}{2}\right) \frac{4^n \pi ^{4 n}}{(4 n+1)!} \,,
		\]
		where \( B_k(t) \) is the \( k \)-th Bernoulli polynomial.

		\begin{proof}
			First recall from \cite[Section 6.1]{HI} the following generating series expressions for \( \zeta(\{4\}^k) \) and \( \zeta^\star(\{4\}^{k}) \),
			\begin{align*}
			Z_4(t) &{}\coloneqq \sum_{k \geq 0} \zeta(\{4\}^k) \, t^k = \frac{\cosh(\sqrt{2} \pi \sqrt[4]{t}) - \cos(\sqrt{2} \pi \sqrt[4]{t})}{2 \pi^2 \sqrt{t}} \,, \\
			Z_4^\star(t) &{}\coloneqq \sum_{k \geq 0} \zeta^\star(\{4\}^k) \, t^k = \frac{\pi^2 \sqrt{t}}{\sin(\pi \sqrt[4]{t}) \sinh(\pi \sqrt[4]{t})} \, ,
			\end{align*}
			which are linked via \( Z_4(-t) Z_4^\star(t) = 1 \).  It follows from the evaluation in Theorem \ref{thm:SSstar} that
			\[
			S_{1,3}^\star(n) = [t^n] \big( Z_4(t) Z_4^\star(\tfrac{1}{4} t) \big) \,.
			\]
			One can check directly that \( Z_4(t) Z_4^\star(\frac{1}{4} t) \) is expressible as
			\[
			\frac{i}{2} \left(
			f\left(\frac{1-i}{2},i \sqrt{2} \pi \sqrt[4]{t}\right)
			+f\left(\frac{1-i}{2},\sqrt{2} \pi 
			\sqrt[4]{t}\right)
			-f\left(\frac{1+i}{2},i \sqrt{2} \pi 
			\sqrt[4]{t}\right)
			-f\left(\frac{1+i}{2},\sqrt{2} \pi  \sqrt[4]{t}\right)
			\right) \,,
			\]
			where
			\[
			f(a, x) \coloneqq \frac{e^{a x}}{e^x - 1} = \sum_{k \geq 0} B_{k}(a) \frac{x^{k-1}}{k!} \,
			\]
			is a shifted version of the generating series of the Bernoulli polynomials.  Extracting the coefficient of \( t^n \) corresponds to taking the \( (4n + 1) \)-th term in the expansions of~\( f \), so we obtain
			\[
			[t^n] \big( Z_4(t) Z_4^\star(\tfrac{1}{4} t) \big) =i \left( B_{4n+1}\left(\frac{1 - i}{2}\right) - B_{4n+1}\left(\frac{1+i}{2}\right) \right) \frac{4^n \pi^{4n}}{(4n + 1)!} \,.
			\]
			Under the symmetry \( B_{k}(x) = (-1)^k B_{k}(1-x) \) of Bernoulli polynomials, this simplifies to the expression given in the statement of the Lemma.
		\end{proof}
	\end{lemma}
	
	\begin{proposition} For $n\geq 1$ we have 
		\begin{align*}
		B_{1,3}(n-1) \cdot A_{1,3}(n) - G_{1,3}(n) =  \alpha_n \zeta(\{1,3\}^{2n}) \in \Q \pi^{8n} \,,
		\end{align*}
		where 
		\[
		\alpha_n = 8 i (8n + 1) \binom{8n}{4n} B_{4n+1}\left(\frac{1-i}{2} \right) \sum_{j=0}^{2n} (-1)^j \big( 1 - 2^{2j-1} \big) \big( 1 - 2^{4n-2j - 1} \big) \binom{4n}{2j} B_{2j} B_{4n-2j} \,.
		\]
		Here \( B_k \) is the \( k \)-th Bernoulli number, and \( B_k(x) \) the \( k \)-th Bernoulli polynomial.
	\end{proposition}
	\begin{proof}
		Using Corollary \ref{cor:ribbonjacobimzv} with an $A$-type ribbon $R$  gives an outside decomposition $(\theta_1,\theta_2)$ with an $A$-type stair $\theta_1$ and a $B$-type stair $\theta_2$. We then immediately obtain 
		\begin{align*}
		G_{1,3}(n) =  \det\begin{pmatrix}
		A_{1,3}(n) &S_{1,3}(n) \\
		S^\star_{1,3}(n) & B_{1,3}(n-1)
		\end{pmatrix}
		\end{align*}
		and thus $    B_{1,3}(n-1) \cdot A_{1,3}(n) - G_{1,3}(n) =     S_{1,3}(n) S^\star_{1,3}(n) $.  So from Theorem  \ref{thm:SSstar} it already follows that \( B_{1,3}(n-1) \cdot A_{1,3}(n) - G_{1,3}(n) \in \pi^{8n} \Q \). \medskip
		
		The explicit formula for \( \alpha_n \) arises as follows.  From Lemma \ref{lem:stairstareval}, we have
		\[
		S^\star_{1,3}(n) = 2 i B_{4 n+1}\left(\frac{1 - i}{2}\right) \frac{4^n \pi ^{4 n}}{(4 n+1)!} \,.
		\]        
		Theorem \ref{thm:SSstar}, gives us \( S_{1,3}(n) = \frac{1}{4^n} \zeta^\star(\{4\}^n) \).  Using known formulae for the Taylor series of \( \csc(x) \) and \( \operatorname{csch}(x) \) to expand out the generating series of \( \zeta^\star(\{4\}^n) \) in Lemma \ref{lem:stairstareval}, gives
		\[
		\zeta^\star(\{4\}^n) = \frac{4 \pi^{4n}}{(4n)!}  \sum_{j = 0}^{2n} (-1)^j (1 - 2^{2j - 1})(1 - 2^{4n - 2j-1}) \binom{4n}{2j} B_{2j} B_{4n-2j} \,.
		\]
		From \eqref{eq:31formula}, we have
		\[
		\zeta(\{1,3\}^{2n}) = \frac{2\pi^{8n}}{(8n+2)!} \,.
		\]
		From these evaluations and the definition of \( \alpha_n \) as
		\[
		\alpha_n = \frac{B_{1,3}(n-1) \cdot A_{1,3}(n) - G_{1,3}(n)}{\zeta(\{1,3\}^{2n})} = \frac{S_{1,3}(n) S^\star_{1,3}(n)}{\zeta(\{1,3\}^{2n})}  \,,
		\]
		we obtain the required formula.
	\end{proof}
	
	From this formula we list some values of \( \alpha_n \) for small \( n \), to confirm the results in (5.2) and (5.3) of \cite{BY}, and the numerics for the weight 32 case $ n = 4 $ thereafter.
	\begin{center}
		\begin{tabular}{c|ccccc}
		$n$ & $1$ & $2$ & $3$ & $4$ & $5$ \\ \hline
		$\alpha_n$ & $70$ & $1074502$ & $\displaystyle\frac{9656199193420}{21}$ & $2222659435447178310$ & $\displaystyle\frac{766533703696349735861335868}{11}$
		 $\mathclap{\phantom{\Bigg)}}$ % stupid formatting trick 
		\end{tabular}
	\end{center}\medskip
	
	After computing \( \alpha_n \) for larger values of \( n \), we note that the denominator is always appears to be \emph{much} smaller than the numerator, but as \( n \) increases the denominator also increases in a way that is not so clear from the formula.  For $ n = 9 $ we have denominator $ 133 $, for $ n = 15 $ we have denominator $ 1085 $, whereas for $ n = 23 $ we have denominator $ 206283 $.

\end{document}